\documentclass[11pt]{amsart}

\usepackage{fullpage}
\usepackage[utf8]{inputenc}
\usepackage{amsmath,amssymb,amsfonts,amsaddr}

\usepackage[linkcolor=red,colorlinks=true]{hyperref}
\usepackage{cleveref}

\usepackage{biblatex}

\numberwithin{equation}{section}
\newtheorem{theorem}{Theorem}[section]

\newtheorem{conjecture}[theorem]{Conjecture}

\newcommand{\Z}{\mathbb{Z}}
\newcommand{\E}{\mathrm{e}}

\newcommand{\D}[1]{\mathop{\mathrm{d}#1}}

\allowdisplaybreaks

\title{Evaluation of the symmetrized Mordell-Tornheim zeta function}
\author{Przemysław Dobrowolski}
\address{Warsaw, Poland}
\email{pdobrowo@gmail.com}


\begin{document}


\begin{abstract}
In this paper we evaluate the symmetrized Mordell-Tornheim zeta function defined as
\begin{equation*}
\overline{\zeta}_n(w_1, \ldots, w_n) = \sum_{\substack{a_1, \ldots, a_n \in \Z^* \\ a_1 + \ldots + a_n = 0}} \frac{1}{\left| a_1^{w_1} \cdots a_n^{w_n} \right|}
\end{equation*}
where $n \ge 1$ is a positive integer representing \emph{depth} and $w_1, \ldots, w_n \ge 1$ are positive integers representing \emph{weight} $w = w_1 + \ldots + w_n$ of the function.

An equivalent formulation is
\begin{equation*}
\overline{\zeta}_{n+1}(w_1, \ldots, w_{n+1}) = \sum_{\substack{a_1, \ldots, a_n \in \Z^* \\ \sum_{k=1}^n a_k \ne 0}} \frac{1}{\left| a_1^{w_1} \cdots a_n^{w_n} \left( a_1 + \ldots + a_n \right)^{w_{n+1}} \right|}
\end{equation*}
which shows the function represents the fully symmetrized case of the Mordell-Tornheim zeta function:
\begin{equation*}
\zeta_{MT,n}(w_1, \ldots, w_n; w_{n+1}) = \sum_{a_1, \ldots, a_n \in \Z^+} \frac{1}{ a_1^{w_1} \cdots a_n^{w_n}(a_1 + \ldots + a_n)^{w_{r+1}}}
\end{equation*}
Compared to the classical Mordell-Tornheim zeta function $\zeta_{MT,n}(w_1, \ldots, w_n; w_{n+1})$ which is restricted to the positive orthant (hyperoctant), the symmetrized one spans the entire $(n-1)$-dimensional hyperplane.

We show that when the depth and the weight of the function are equal, that is for $\overline{\zeta}_n(1, \ldots, 1)$, it has a remarkably simple representation in terms of standard functions:
\begin{equation*}
\overline{\zeta}_n(1, \ldots, 1) = B_n(f^{(1)}(0), \ldots, f^{(n)}(0))
\end{equation*}
where $B_n$ is $n$-th complete exponential Bell polynomial and $f^{(n)}(0)$ is $n$-th derivative at $x=0$ of function $f(x)$ defined as:
\begin{equation*}
f(x) = \ln \binom{-2x}{-x}
\end{equation*}
Additionally, we show the value can be expressed using the following polynomials with positive integer coefficients over the values of zeta function:
\begin{equation*}
\overline{\zeta}_n(1, \ldots, 1) = B_n(0, (2^2 - 2) \Gamma(2) \zeta(2), \ldots, (2^n - 2) \Gamma(n) \zeta(n))
\end{equation*}
or equivalently, over the values of eta function:
\begin{equation*}
\overline{\zeta}_n(1, \ldots, 1) = B_n(0, 2^2 \Gamma(2) \eta(2), \ldots, 2^n \Gamma(n) \eta(n))
\end{equation*}
The list of explicit values for small $1 \le n \le 10$ is available in the appendix.
\end{abstract}

\maketitle


\section{Introduction}
In the recent years, different variants of multiple zeta values (MZV) have been studied extensively. Significant portion of the research contributes to a better understanding of the internal structure of the MZV values. Some of the research has provided the exact values of multiple zeta values for special cases however, generic formulas are not known.

The Mordell-Tornheim zeta function has been introduced by Matsumoto in \cite{Matsumoto2003} as
\begin{equation*}
\zeta_{MT,n}(w_1, \ldots, w_n; w_{n+1})=\sum_{a_1, \ldots, a_n \in \Z^+} \frac{1}{ a_1^{w_1} \cdots a_n^{w_n}(a_1 + \ldots + a_n)^{w_{r+1}}}
\end{equation*}
based on the classical papers from Tornheim \cite{Tornheim1950} and Mordell \cite{Mordell1958}. The function was initially evaluated for the depth 2 (double zeta function) in \cite{Matsumoto2003} and later extended for the depth 3 (triple zeta-function) in \cite{MatsumotoEtal2008}. The generic case of $r$-ple zeta-function was discussed in \cite{MatsumotoTsumura2006} as a special case of Witten multiple zeta-function.

A symmetric variant of Mordell-Tornheim function was introduced and evaluated by Bachmann in \cite{Bachmann2021}. The author obtained a new formula by summing the classical Mordell-Tornheim function over all permutations of the argument indices.

In this paper we introduce and evaluate a different symmetrized variant of Mordell-Tornheim function. Instead of summing up the classical Mordell-Tornheim functions, the new function spans the whole domain only excluding the singularities. The result unveils that the values can be expressed as the values of polynomials with non-negative integer coefficients over consecutive zeta or eta values.

\section{Evaluation of the symmetrized Mordell-Tornheim zeta function}

Let the symmetrized Mordell-Tornheim zeta function $\overline{\zeta}_n(w_1, \ldots, w_n)$ be defined as
\begin{equation}\label{eqn_generic_def}
\overline{\zeta}_n(w_1, \ldots, w_n) = \sum_{\substack{a_1, \ldots, a_n \in \Z^* \\ \sum_{k=1}^n a_k = 0}} \frac{1}{\left| \prod_{k=1}^n a_k^{w_k} \right|}
\end{equation}
for $n \ge 1$ and $w_1, \ldots, w_n \ge 1$.

By moving the condition $\sum_{k=1}^n a_k = 0$ into the sum as $a_n = -\sum_{k=1}^{n-1} a_k \in \Z^*$, the function becomes
\begin{equation}\label{eqn_generic_def_expanded}
\overline{\zeta}_n(w_1, \ldots, w_n) = \sum_{\substack{a_1, \ldots, a_{n-1} \in \Z^* \\ \sum_{k=1}^{n-1} a_k \ne 0}} \frac{1}{ \left| \prod_{k=1}^{n-1} a_k^{w_k} \left( \sum_{k=1}^{n-1} a_k \right)^{w_n} \right|}
\end{equation}
which shows that compared to Mordell-Tornheim zeta function, the newly introduced function extends the summation domain to the whole $n-1$ dimensional subspace excluding the singularities and it sums the absolute values of the elements instead of the signed values.

The first step is to determine the convergence of the sum for all values of $w_k$.

\begin{theorem}
The symmetrized Mordell-Tornheim function $\overline{\zeta}_n(w_1, \ldots, w_n)$ converges for all parameters $w_1, \ldots, w_n \ge 1$.
\end{theorem}

\begin{proof}
We consider the case $n=1$ separately. There is:
\begin{equation*}
\overline{\zeta}_n(w_1) = \sum_{\substack{a_1 \in \Z^* \\ a_1 = 0}} \frac{1}{\left| a_1^{w_1}\right|} = 0
\end{equation*}
which implies that the sum converges. Now we consider $n \ge 2$. It is enough to verify the convergence for $w_1 = \ldots = w_k = 1$ which is the upper bound for all the other values of parameters $w_k$:
\begin{equation*}
\overline{\zeta}_n(w_1, \ldots, w_n) = \sum_{\substack{a_1, \ldots, a_n \in \Z^* \\ \sum_{k=1}^n a_k = 0}} \frac{1}{\left| \prod_{k=1}^n a_k^{w_k} \right|} \le \sum_{\substack{a_1, \ldots, a_n \in \Z^* \\ \sum_{k=1}^n a_k = 0}} \frac{1}{\left| \prod_{k=1}^n a_k \right|}
\end{equation*}
Let $\Omega = \left\{ a_1, \ldots, a_n \in \Z^*: \quad \sum_{k=1}^n a_k = 0 \right\}$ and $a = (a_1, \ldots, a_k) \in \Omega$ be an element which belongs to the sum. For each $a$ we rewrite the absolute value of the product in the denominator as the product of the absolute values. Then, we sort the indices of $a$ according to their absolute value in a non-decreasing order: $|a_1| \le \ldots \le |a_n|$. There are at most $n!$ permutations of the indices of $a$ which all provide the same value of the summed element. The number of permutations can be smaller in case of duplicates among indices of $a$. Hence, the following upper bound holds:
\begin{equation*}
\sum_{\substack{a_1, \ldots, a_n \in \Z^* \\ \sum_{k=1}^n a_k = 0}} \frac{1}{\left| \prod_{k=1}^n a_k \right|} \le n! \sum_{\substack{a \in \Omega \\ |a_1| \le \ldots \le |a_n|}} \frac{1}{\prod_{k=1}^n |a_k|}
\end{equation*}
By definition, element $a_n$ is not smaller than $a_{n-1}$ hence $\frac{1}{a_n}$ is not greater than $\frac{1}{a_{n-1}}$. We can eliminate element $a_n$ from the sum completely and drop the condition $a_n = -\sum_{k=1}^{n-1} a_k$ to obtain the following upper bound:
\begin{equation*}
n! \sum_{\substack{a \in \Omega \\ |a_1| \le \ldots \le |a_n|}} \frac{1}{\prod_{k=1}^n |a_k|} \le n! \sum_{\substack{a_1, \ldots, a_{n-1} \in \Z^* \\ |a_1| \le \ldots \le |a_{n-1}|}} \frac{1}{|a_1| |a_2| \cdots |a_{n-2}||a_{n-1}|^2}
\end{equation*}
For each of $2^{n-1}$ combinations of the signs of $a_k$, the value of the summed element remains the same. Therefore, we can take only the positive elements $a_k$ and multiply the sum by $2^{n-1}$:
\begin{equation*}
n! \sum_{\substack{a_1, \ldots, a_{n-1} \in \Z^* \\ |a_1| \le \ldots \le |a_{n-1}|}} \frac{1}{|a_1| |a_2| \cdots |a_{n-2}||a_{n-1}|^2} = n! 2^{n-1} \sum_{a_1=1}^\infty \sum_{a_2=a_1}^\infty \ldots \sum_{a_{n-1}=a_{n-2}}^\infty \frac{1}{a_1 a_2 \cdots a_{n-2} a_{n-1}^2}
\end{equation*}
Using the following estimate
\begin{equation*}
\sum_{k=p}^\infty \frac{1}{k^2} \le \frac{1}{p^2} + \int_p^\infty \frac{1}{t^2} \D{t} = \frac{1}{p^2} + \frac{1}{p} \le \frac{2}{p}
\end{equation*}
we simplify further by evaluating sums one by one starting from the inner sum to the most outer sum which is evaluated exactly:
\begin{align*}
& n! 2^{n-1} \sum_{a_1=1}^\infty \sum_{a_2=a_1}^\infty \ldots \sum_{a_{n-1}=a_{n-2}}^\infty \frac{1}{a_1 a_2 \cdots a_{n-2} a_{n-1}^2} = n! 2^{n-1} 2^1 \sum_{a_1=1}^\infty \sum_{a_2=a_1}^\infty \ldots \sum_{a_{n-2}=a_{n-3}}^\infty \frac{1}{a_1 \cdots a_{n-2}^2} \\ & = \cdots = n! 2^{n-1} 2^{n-2} \sum_{a_1=1}^\infty \frac{1}{a_1^2}
\end{align*}
The last sum converges to $\zeta(2)$, therefore we have the final upper bound:
\begin{equation*}
\overline{\zeta}_n(w_1, \ldots, w_n) \le n! 2^{2n-3} \zeta(2) < \infty
\end{equation*}
Which completes the proof.
\end{proof}

In this paper we will evaluate the following special case:
\begin{equation}\label{eqn_def}
\overline{\zeta}_n = \overline{\zeta}_n(1, \ldots, 1) = \sum_{\substack{a_1, \ldots, a_n \in \Z^* \\ \sum_{k=1}^n a_k = 0}} \frac{1}{\left| \prod_{k=1}^n a_k \right|}
\end{equation}

Which by \eqref{eqn_generic_def_expanded} is equal to
\begin{equation}\label{eqn_def_expanded}
\overline{\zeta}_n = \sum_{\substack{a_1, \ldots, a_{n-1} \in \Z^* \\ \sum_{k=1}^{n-1} a_k \ne 0}} \frac{1}{\left| \prod_{k=1}^{n-1} a_k \right| \left| \sum_{k=1}^{n-1} a_k \right|}
\end{equation}


\begin{theorem}\label{thm_main}
When the depth and the weight of the symmetrized Mordell-Tornheim function are equal, that is for $\overline{\zeta}_n(1, \ldots, 1)$, it can be evaluated as
\begin{equation*}
\overline{\zeta}_n(1, \ldots, 1) = B_n(f^{(1)}(0), \ldots, f^{(n)}(0))
\end{equation*}
where $B_n$ is $n$-th complete exponential Bell polynomial and $f^{(n)}(0)$ is $n$-th derivative at $x=0$ of function $f(x)$ defined as:
\begin{equation*}
f(x) = \ln \binom{-2x}{-x}
\end{equation*}
Additionally, the value can be expressed using the following polynomials with positive integer coefficients over the values of zeta function:
\begin{equation*}
\overline{\zeta}_n(1, \ldots, 1) = B_n(0, (2^2 - 2) \Gamma(2) \zeta(2), \ldots, (2^n - 2) \Gamma(n) \zeta(n))
\end{equation*}
or equivalently, over the values of eta function:
\begin{equation*}
\overline{\zeta}_n(1, \ldots, 1) = B_n(0, 2^2 \Gamma(2) \eta(2), \ldots, 2^n \Gamma(n) \eta(n))
\end{equation*}.
\end{theorem}

\begin{proof}
We start by noting that in the sum \eqref{eqn_def_expanded} many elements are equal. In particular, for all combinations of the signs of $a_k$, the absolute value of the product in the denominator remains the same. In the case of the absolute value of the sum in the denominator, it does not change as long as the number $p$ of positive or equivalently, the number $m$ negative elements $a_k$, remains the same. We split the sum according to the number $p$ of positive values of $a_k$. 

The number of elements with $p$ positive signs is equal to $\binom{n-1}{p}$. With this observation we can write
\begin{equation*}
\overline{\zeta}_n = \sum_{p=0}^{n-1} \binom{n-1}{p} \sum_{\substack{l_1, \ldots, l_{n-1} \in \Z^+ \\ s_1, \ldots, s_{p} = 1 \\ s_{p+1}, \ldots, s_{n-1} = -1 \\ \sum_{k=1}^{n-1} s_k l_k \ne 0}} \frac{1}{\left(\prod_{k=1}^{n-1} l_k \right) \left| \sum_{k=1}^{n-1} s_k l_k \right|}
\end{equation*}
where $a_k = s_k l_k$ and vector $s$ is the \emph{signature} of an element $a$.

The remaining absolute value in the denominator can be reduced by observing that an element with a signature $s$ is included in the sum if and only if the corresponding element with signature $-s$ is also included. In other words, the sum is symmetric by the signature of the elements.

Next, for each $l$, either $s$ or $-s$ signature element is selected so that the sum $\sum_{k=1}^{n-1} s_k l_k$ is always positive. Hence, there is
\begin{equation*}
\overline{\zeta}_n = 2 \sum_{p=0}^{n-1} \binom{n-1}{p} \sigma(p, n-1-p)
\end{equation*}
where the new function $\sigma$ is defined as
\begin{equation}\label{eqn_sigma_def}
\sigma(p, m) = \sum_{\substack{l_1, \ldots, l_{p+m} \in \Z^+ \\ s_1, \ldots, s_{p} = 1 \\ s_{p+1}, \ldots, s_{p+m} = -1 \\ \sum_{k=1}^{p+m} s_k l_k > 0}} \frac{1}{\left(\prod_{k=1}^{p+m} l_k \right) \left( \sum_{k=1}^{p+m} s_k l_k \right)}
\end{equation}

We calculate the boundary cases first. It is easy to find the value of $\sigma(0, m)$:
\begin{equation*}
\sigma(0, m) = \sum_{\substack{l_1, \ldots, l_m \in \Z^+ \\ s_1, \ldots, s_m = -1 \\ \sum_{k=1}^m s_k l_k > 0}} \frac{1}{\left(\prod_{k=1}^m l_k \right) \left( \sum_{k=1}^m s_k l_k \right)}
\end{equation*}
For $s_k = -1$ the sum $\sum_{k=1}^m s_k l_k < 0$ which contradicts the condition $\sum_{k=1}^m s_k l_k > 0$ and hence the sum is equal to zero:
\begin{equation}\label{eqn_sigma_0_m}
\sigma(0, m) = 0
\end{equation}

For the opposite case $\sigma(p, 0)$, there is:
\begin{equation*}
\sigma(p, 0) = \sum_{\substack{l_1, \ldots, l_p \in \Z^+ \\ s_1, \ldots, s_p = 1 \\ \sum_{k=1}^p s_k l_k > 0}} \frac{1}{\left(\prod_{k=1}^p l_k \right) \left( \sum_{k=1}^p s_k l_k \right)} = \sum_{l_1, \ldots, l_p \in \Z^+} \frac{1}{\left(\prod_{k=1}^p l_k \right) \left( \sum_{k=1}^p l_k \right)}
\end{equation*}
It can be calculated as follows:
\begin{align*}
& \sum_{l_1, \ldots, l_p \in \Z^+} \frac{1}{\left(\prod_{k=1}^p l_k \right) \left( \sum_{k=1}^p l_k \right)} = \sum_{l_1, \ldots, l_p \in \Z^+} \frac{1}{\prod_{k=1}^p l_k} \int_0^1 t^{\left(\sum_{k=1}^p l_k\right) - 1} \D{t} \\
& = \int_0^1 \frac{1}{t} \left( \sum_{l_1 \in \Z^+} \frac{t^{l_1}}{l_1} \right) \cdots \left( \sum_{l_p \in \Z^+} \frac{t^{l_p}}{l_p} \right) \D{t} = \int_0^1 \frac{1}{t} \left(-\ln(1-t)\right)^p \D{t} \\
& = \int_0^{\infty} \frac{\E^{-x}}{1 - \E^{-x}} x^p \D{t} = \int_0^{\infty} \frac{x^p}{\E^x - 1} \D{t}
\end{align*}
which is a well known integral equal to
\begin{equation}\label{eqn_sigma_p_0}
\sigma(p, 0) = \Gamma(p+1)\zeta(p+1)
\end{equation}

Next, assuming $p > 0$ and $m > 0$, we calculate the generic case. We separate the positive and the negative signs in the signature of $\sigma$:
\begin{equation}\label{eqn_sigma_separated}
\sigma(p, m) = \sum_{\substack{l_1, \ldots, l_{p+m} \in \Z^+ \\ \sum_{k=1}^{p} l_k - \sum_{k=p+1}^{p+m} l_k > 0}} \frac{1}{\left(\prod_{k=1}^{p+m} l_k \right) \left( \sum_{k=1}^p l_k - \sum_{k=p+1}^{p+m} l_k \right)}
\end{equation}

Let $d$ be the equal to the difference in the following sums:
\begin{equation}\label{eqn_d_def}
d = \sum_{k=1}^p l_k - \sum_{k=p+1}^{p+m} l_k
\end{equation}
Equation \eqref{eqn_d_def} increases the number of unknowns which is compensated by additional conditions $d > 0$ and $\sum_{k=1}^p l_k = d + \sum_{k=p+1}^{p+m} l_k$, hence:
\begin{equation}\label{eqn_sigma_separated_simplified}
\sigma(p, m) = \sum_{\substack{l_1, \ldots, l_{p+m}, d \in \Z^+ \\ \sum_{k=1}^p l_k = d + \sum_{k=p+1}^{p+m} l_k}} \frac{1}{\left(\prod_{k=1}^{p+m} l_k \right) d}
\end{equation}

It can be seen that all $l_k$ and $d$ play the same role in the formula and within the sum condition, LHS and RHS sums are independent in respect of the $l_k$ subsets they involve. Introduce a summation index $s$, equal to the value of both LHS and RHS:
\begin{equation*}
s = \sum_{k=1}^p l_k = d + \sum_{k=p+1}^{p+m} l_k
\end{equation*}
The smallest achievable value of $s$ is $1$ when $d = 1$ and $m = 0$. We split the sum into a product of sums according to the summation index $s \in \Z^+$, assuming that LHS and RHS are independent and the inner product is also split:
\begin{equation*}
\sigma(p, m) = \sum_{s=1}^{\infty} \left( \sum_{\substack{l_1, \ldots, l_p \in \Z^+ \\ \sum_{k=1}^p l_k = s}} \frac{1}{\prod_{k=1}^{p+m} l_k } \right) \left( \sum_{\substack{l_{p+1}, \ldots, l_{p+m}, d \in \Z^+ \\ \sum_{k=p+1}^{p+m} l_k = s}} \frac{1}{\left(\prod_{k=p+1}^{p+m} l_k \right) d} \right)
\end{equation*}
It is now possible recognize that both sums are cases of the same generic sum $h_n(s)$:
\begin{equation*}
h_n(s) = \sum_{\substack{l_1, \ldots, l_n \in \Z^+ \\ \sum_{k=1}^n l_k = s}} \frac{1}{\prod_{k=1}^n l_k}
\end{equation*}
which implies:
\begin{equation*}
\sigma(p, m) = \sum_{s=1}^{\infty} h_p(s) h_{m+1}(s)
\end{equation*}

Next, we evaluate sums $h_n(s)$ using generating functions. We introduce a generating function $G_n(s)$ which coefficients span the values of $h_n(s)$ for the consecutive powers $x^s$:
\begin{equation}\label{eqn_gen_g_def}
G_n(x) = \sum_{s=n}^{\infty} h_n(s) x^s
\end{equation}
In the definition of $G_n$ the sum starts from index $n$ because all the previous coefficients are equal to $0$. Inserting the definition of $h_n(s)$ one gets:
\begin{equation*}
G_n(s) = \sum_{s=n}^{\infty} \left( \sum_{\substack{l_1, \ldots, l_n \in \Z^+ \\ \sum_{k=1}^n l_k = s}} \frac{1}{\prod_{k=1}^n l_k} \right) x^s
\end{equation*}
The condition $\sum_{k=1}^n l_k = s$ from the inner sum implies that the most of the coefficients of $G_n(s)$ are equal to zero, hence:
\begin{equation*}
G_n(s) = \sum_{l_1, \ldots, l_n \in \Z^+} \frac{1}{\prod_{k=1}^n l_k} x^{\sum_{k=1}^n l_k}
\end{equation*}
This sum can be further simplified
\begin{align*}
G_n(s) & = \sum_{l_1, \ldots, l_n \in \Z^+} \frac{x^{l_1}}{l_1} \cdots \frac{x^{l_n}}{l_n} = \left( \sum_{l_1 \in \Z^+} \frac{x^{l_1}}{l_1} \right) \cdots \left( \sum_{l_n \in \Z^+} \frac{x^{l_n}}{l_n} \right) \\
& = \left( \sum_{l \in \Z^+} \frac{x^{l}}{l} \right)^n = \left( -\ln(1 - x) \right)^n
\end{align*}

\section{Use of unsigned Stirling numbers of the first kind}

A definition of unsigned String numbers of the first kind is
\begin{equation}\label{eqn_stirling_first_kind}
\frac{1}{n!}\left(-\ln(1 - x)\right)^n = \sum_{s=n}^{\infty} {s \brack n } \frac{x^s}{s!}
\end{equation}
Using the above definition, we can write
\begin{equation*}
h_n(s) = \frac{n!}{s!} { s \brack n }
\end{equation*}
Which further implies that
\begin{equation}\label{eqn_sigma_final}
\sigma(p, m) = \sum_{s=1}^{\infty} h_p(s) h_{m+1}(s) = p! (q+1)! \sum_{s=1}^{\infty} \frac{{ s \brack p }{ s \brack m + 1 }}{(s!)^2}
\end{equation}
The above formula was derived for $p > 0$ and $m > 0$. The two missing cases will be verified against the previously evaluated special cases.

For $p = 0$, there is
\begin{equation*}
\forall_{n \in \Z^+} {n \brack 0} = 0
\end{equation*}
so the sum zeroes and matches the previously calculated value \eqref{eqn_sigma_0_m}.

Similarly, for $m = 0$, there is
\begin{equation*}
\sigma(p, 0) = p! \sum_{s=1}^{\infty} \frac{{ s \brack p }{ s \brack 1 }}{(s!)^2}
\end{equation*}
using ${s \brack 1} = (s - 1)!$ one simplifies further
\begin{equation*}
\sigma(p, 0) = p! \sum_{s=1}^{\infty} \frac{{ s \brack p } (s - 1)!}{(s!)^2} = p! \sum_{s=1}^{\infty} \frac{{ s \brack p } }{s (s!)}
\end{equation*}
Using a classical result on Stirling numbers of the second kind from \cite[pp.194, eqn. 11]{Jordan1950}:
\begin{equation*}
\sum_{s=1}^{\infty} \frac{{ s \brack p } }{s (s!)} = \zeta(p + 1)
\end{equation*}
we obtain
\begin{equation*}
\sigma(p, 0) = \Gamma(p + 1) \zeta(p + 1)
\end{equation*}
which also matches the previously calculated \eqref{eqn_sigma_p_0}. In the result, formula \eqref{eqn_sigma_final} is well defined for all $p \ge 0$ and $m \ge 0$.

Using the formula \eqref{eqn_sigma_final}, we rewrite the main sum of $\overline{\zeta}_n$ as
\begin{align*}
\overline{\zeta}_n & = 2 \sum_{p=0}^{n-1} \binom{n-1}{p} \sigma(p, n-1-p) \\ 
& = 2 \sum_{p=0}^{n-1} \binom{n-1}{p} p! (n-1-p+1)! \sum_{s=1}^{\infty} \frac{{ s \brack p }{ s \brack n-1-p + 1 }}{(s!)^2} \\
& = 2 (n-1)! \sum_{p=0}^{n-1} (n-p) \sum_{s=1}^{\infty} \frac{{ s \brack p }{ s \brack n-p }}{(s!)^2} \\
& = 2 (n-1)! \sum_{s=1}^{\infty} \frac{1}{(s!)^2} \sum_{p=0}^{n-1} (n-p) { s \brack p }{ s \brack n-p }
\end{align*}
In the inner sum, we add the missing element $p = n$ for which the value is zero. Then, we make the formula symmetrized by reversing the order of summation:
\begin{equation*}
\sum_{p=0}^n (n-p) { s \brack p }{ s \brack n-p } = \sum_{p=0}^n p { s \brack p }{ s \brack n-p } = \frac{n}{2} \sum_{p=0}^n { s \brack p }{ s \brack n-p }
\end{equation*}
Therefore, the symmetrized formula for $\overline{\zeta}_n$ becomes:
\begin{equation*}
\overline{\zeta}_n = n! \sum_{s=1}^{\infty} \frac{1}{(s!)^2} \sum_{p=0}^n { s \brack p }{ s \brack n-p }
\end{equation*}

\section{Use of Pochhammer symbol}

Let $c_n(s)$ be defined as
\begin{equation}\label{eqn_c_def}
c_n(s) = \sum_{p=0}^n { s \brack p }{ s \brack n-p }
\end{equation}

Unsigned Stirling numbers of the first kind are related to the Pochhammer symbol (rising factorial) via the following generating function
\begin{equation*}
x^{(s)} = \sum_{k=0}^s { s \brack k } x^k
\end{equation*}

We start by evaluating the squared Pochhammer symbol as
\begin{align*}
\left( x^{(s)} \right)^2 & = \left( \sum_{k=0}^s { s \brack k } x^k \right) \left( \sum_{l=0}^s { s \brack l } x^l \right) = \sum_{k=0}^{s+s} \left( \sum_{p=0}^k { s \brack p } { s \brack k - p } x^p x^{k - p} \right) \\
& = \sum_{k=0}^{2 s} \left( \sum_{p=0}^k { s \brack p } { s \brack k - p } \right) x^k = \sum_{k=0}^{2 s} c_k(s) x^k
\end{align*}
We see that the square of Pochhammer symbol is precisely the generating function for values $c_n(s)$, hence
\begin{equation}\label{eqn_c_as_poch_sqr}
c_n(s) = \left[ x^n \right] \left( x^{(s)} \right)^2
\end{equation}
Using the above formula, we further simplify $\overline{\zeta}_n$ to
\begin{equation*}
\overline{\zeta}_n = n! \sum_{s=1}^{\infty} \frac{1}{(s!)^2} \left[ x^n \right] \left( x^{(s)} \right)^2 = n! \sum_{s=1}^{\infty} \left[ x^n \right] \left( \frac{x^{(s)}}{s!} \right)^2
\end{equation*}
To sum the generating functions and extract $\left[ x^n \right]$ outside of the sum, we first check the convergence domain where the swap is possible. There is:
\begin{equation*}
\left( \frac{x^{(s)}}{s!} \right)^2 = \left(\frac{x(x+1)(x+2)\cdots(x+s-1)}{s!}\right)^2
\end{equation*}
It is enough to select a domain where the growth rate is of the order smaller than $O(s^{-2})$. Taking $x = \frac{1}{2}$, the numerator becomes a function of order $\frac{s!!}{2^s}$. Diving the value by denominator one gets:
\begin{equation*}
\left( \frac{ \frac{s!!}{2^s} }{s!} \right)^2 = \frac{1}{4^s(s-1)!} < \frac{1}{s^2}
\end{equation*}
Assuming domain $0 \le x < \frac{1}{2}$, one can write
\begin{equation}\label{eqn_zn_in_domain}
\overline{\zeta}_n = n! \left[ x^n \right] \sum_{s=1}^{\infty} \left( \frac{x^{(s)}}{s!} \right)^2, \quad |x| < \frac{1}{2}
\end{equation}
The series can be extended to also include element $s=0$ because the additional element does not change any of the coefficients for $ n \ge1$. We rewrite the series using Pochhammer symbols:
\begin{equation}\label{eqn_zn_almost_2f1}
\overline{\zeta}_n = n! \left[ x^n \right] \sum_{s=0}^{\infty} \left( \frac{x^{(s)}}{s!} \right)^2 = n! \left[ x^n \right] \sum_{s=0}^{\infty} \frac{x^{(s)} x^{(s)}}{1^{(s)}} \frac{1^s}{s!}
\end{equation}
Formula \eqref{eqn_zn_almost_2f1} is a special case of the hypergeometric series ${_2F_1}$:
\begin{equation*}
\overline{\zeta}_n = n! \left[ x^n \right] {_2F_1}(x;x;1;1)
\end{equation*}
Using Gauss hypergeometric theorem \cite[section 2.2]{Andrews1999}
\begin{equation*}
{_2F_1}(a;b;c;z) = \frac{\Gamma(c)\Gamma(c-a-b)}{\Gamma(c-a)\Gamma(c-b)}, \quad \Re(c) > \Re(a+b)
\end{equation*}
we simplify formula \eqref{eqn_zn_almost_2f1} to:
\begin{equation}\label{eqn_final_gen}
\overline{\zeta}_n = n! \left[ x^n \right] \frac{\Gamma(1) \Gamma(1-x-x)}{\Gamma(1-x)\Gamma(1-x)} = n! \left[ x^n \right] \frac{\Gamma(1-2x)}{\Gamma(1-x)^2}
\end{equation}
which is satisfied in the convergence domain \eqref{eqn_zn_in_domain} because $\Re(1) > \Re(\frac{1}{2} - \epsilon + \frac{1}{2} - \epsilon)$.

Our next step is to evaluate the coefficients of the generating function $\frac{\Gamma(1-2x)}{\Gamma(1-x)^2}$. We expand the function into the Taylor series around $x = 0$. To simplify the calculation, instead of directly expanding the generating function, first we expand the logarithm of the function and afterwards we take the exponential of the result.

There is
\begin{equation}\label{eqn_gen_log}
\ln \left(\frac{\Gamma(1-2x)}{\Gamma(1-x)^2}\right) = \ln \left(\Gamma(1-2x)\right) - 2 \ln \left(\Gamma(1-x)\right)
\end{equation}
Using the known \cite[section 1.2]{Andrews1999} Taylor series for $\ln \left( \Gamma(1-x) \right)$
\begin{equation*}
\ln \left( \Gamma(1-x) \right) = \gamma x + \frac{\zeta(2)}{2} x^2 + \frac{\zeta(3)}{3} x^3 + \ldots + \frac{\zeta(k)}{k} x^k + \ldots
\end{equation*}
we expand formula \eqref{eqn_gen_log} as
\begin{align}\label{eqn_gen_before_exp}
& \ln \left(\frac{\Gamma(1-2x)}{\Gamma(1-x)^2}\right) = \gamma 2 x + \frac{\zeta(2)}{2} (2x)^2 + \frac{\zeta(3)}{3} (2x)^3 + \ldots \\
& - 2 \left( \gamma x + \frac{\zeta(2)}{2} x^2 + \frac{\zeta(3)}{3} x^3 + \ldots \right) = \sum_{k=2}^{\infty} \left(2^k - 2\right) \frac{\zeta(k)}{k} x^k \notag
\end{align}
Next, we calculate the exponential of \eqref{eqn_gen_before_exp} as
\begin{equation}\label{eqn_gen_after_exp}
\exp \left( \sum_{k=2}^{\infty} \left(2^k - 2\right) \frac{\zeta(k)}{k} x^k \right) = \exp \left( \sum_{k=2}^{\infty} \left(2^k - 2\right) \Gamma(k) \zeta(k) \frac{x^k}{k!} \right)
\end{equation}
Then, using the definition of $n$-th complete exponential Bell polynomial
\begin{equation*}
\exp \left( \sum_{k=1}^{\infty} c_k \frac{x^k}{k!} \right) = \sum_{n=0}^{\infty} B_n(c_1, \ldots, c_n) \frac{x^n}{n!}
\end{equation*}
and the following coefficients
\begin{align*}
c_1 & = 0 \\
c_{k \ge 1} & = \left(2^k - 2\right) \Gamma(k) \zeta(k)
\end{align*}
we turn \eqref{eqn_gen_after_exp} into the following generating function
\begin{equation*}
\frac{\Gamma(1-2x)}{\Gamma(1-x)^2} = B_n(0, \left(2^2 - 2\right) \Gamma(2) \zeta(2), \ldots, \left(2^n - 2\right) \Gamma(n) \zeta(n)) \frac{x^n}{n!}
\end{equation*}
By plugging formula \eqref{eqn_gen_after_exp} into \eqref{eqn_final_gen}, we finally get the final equation by extracting the $n$-th coefficient:
\begin{align}\label{eqn_final_zeta}
\overline{\zeta}_n & = n! \left[ x^n \right] B_n(0, \left(2^2 - 2\right) \Gamma(2) \zeta(2), \ldots, \left(2^n - 2\right) \Gamma(n) \zeta(n)) \frac{x^n}{n!} \notag \\
& = B_n(0, \left(2^2 - 2\right) \Gamma(2) \zeta(2), \ldots, \left(2^n - 2\right) \Gamma(n) \zeta(n))
\end{align}
Using the relation $2^n \eta(n) = (2^n - 2) \zeta(n)$ we can also provide an alternative formulation
\begin{equation}\label{eqn_final_eta}
\overline{\zeta}_n = B_n(0, 2^2 \Gamma(2) \eta(2), \ldots, 2^n \Gamma(n) \eta(n))
\end{equation}

An additional expression can be found by noting that the arguments of Bernoulli polynomial are related to the generating function discussed above and can be extracted by evaluating the derivative of the generating function. Let $f(x)$ be defined as
\begin{align*}
f(x) & = \ln \left(\frac{\Gamma(1-2x)}{\Gamma(1-x)^2}\right) \\
& = 0 x^0 + 0 x^1 + \sum_{k=2}^{\infty} \left(2^k - 2\right) \Gamma(k) \zeta(k) \frac{x^k}{k!}
\end{align*}
To extract $k$-th coefficient, one can calculate the following derivative:
\begin{equation*}
c_k = \frac{f^{(k)}(0)}{k!}
\end{equation*}
which for $k \ge$ matches the values of the previously calculated coefficients $\left(2^k - 2\right) \Gamma(k) \zeta(k)$. The case $k=1$ needs to be calculated separately:
\begin{align*}
f^{(1)}(0) & = \lim_{x \to 0} \frac{\D{}}{\D{x}} \ln \left(\frac{\Gamma(1-2x)}{\Gamma(1-x)^2}\right) = \lim_{x \to 0} \frac{\D{}}{\D{x}} \left( \ln \Gamma(1-2x) - 2 \ln \Gamma(1-x)\right) \\
& = 2 \lim_{x \to 0} \left(\frac{\Gamma^{'}(1-x)}{\Gamma(1-x)} - \frac{\Gamma^{'}(1-2x)}{\Gamma(1-2x)}\right) = 2 \left(\frac{\Gamma^{'}(1)}{\Gamma(1)} - \frac{\Gamma^{'}(1)}{\Gamma(1)}\right) = 0
\end{align*}
Therefore one can write
\begin{equation}\label{eqn_final_alt}
\overline{\zeta}_n = B_n(f^{(1)}(0), \ldots, f^{(n)}(0))
\end{equation}
Function $f(x)$ can be further simplified as
\begin{equation*}
f(x) = \ln \left( \frac{\Gamma(1-2x)}{\Gamma(1-x)^2} \right) = \ln \left( \frac{\Gamma(-2x + 1)}{\Gamma(-x + 1)\Gamma(-2x - (-x) + 1)} \right) = \ln \binom{-2 x}{-x}
\end{equation*}
which completes the proof of the main theorem.
\end{proof}

\section{Conclusions}
Evaluation of the symmetrized Mordell-Tornheim function shows that the values are equal to the valuation of the complete exponential Bernoulli polynomial over the derivatives of the logarithm of the central binomial coefficient. This result suggests that more similar concise relations may exist for the values of the standard Mordell-Tornheim function or other MZV-like formulas.

\appendix

\section{Explicit values of $\overline{\zeta}_n(1, \ldots, 1)$ for small $n$}

The complete expansion of the Bell polynomials leads to the following explicit formulas for small values of $n$:
\begin{align*}
& \overline{\zeta}_1 = \sum_{\substack{a_1 \ne 0 \\ a_1 = 0}} \frac{1}{\left| a_1 \right|}
=
0 \\
& \overline{\zeta}_2 = \sum_{\substack{a_1, a_2 \ne 0 \\ a_1 + a_2 = 0}} \frac{1}{\left| a_1 a_2 \right|}
=
2 \zeta(2) \\
& \overline{\zeta}_3 = \sum_{\substack{a_1, a_2, a_3 \ne 0 \\ a_1 + a_2 + a_3 = 0}} \frac{1}{\left| a_1 a_2 a_3 \right|}
=
12 \zeta(3) \\
& \overline{\zeta}_4 = \sum_{\substack{a_1, \ldots, a_4 \ne 0 \\ a_1 + \ldots + a_4 = 0}} \frac{1}{\left| a_1 \cdots a_4 \right|}
=
12 \Big( \zeta(2)^2 + 7 \zeta(4) \Big) \\
& \overline{\zeta}_5 = \sum_{\substack{a_1, \ldots, a_5 \ne 0 \\ a_1 + \ldots + a_5 = 0}} \frac{1}{\left| a_1 \cdots a_5 \right|}
=
240 \Big( \zeta(2) \zeta(3) + 3 \zeta(5) \Big) \\
& \overline{\zeta}_6 = \sum_{\substack{a_1, \ldots, a_6 \ne 0 \\ a_1 + \ldots + a_6 = 0}} \frac{1}{\left| a_1 \cdots a_6 \right|}
=
120 \Big( \zeta(2)^3 + 12 \zeta(3)^2 + 21 \zeta(2) \zeta(4) + 62 \zeta(6) \Big) \\
& \overline{\zeta}_7 = \sum_{\substack{a_1, \ldots, a_7 \ne 0 \\ a_1 + \ldots + a_7 = 0}} \frac{1}{\left| a_1 \cdots a_7 \right|}
=
5040 \Big( \zeta(2)^2 \zeta(3) + 7 \zeta(3) \zeta(4) + 6 \zeta(2) \zeta(5) + 18 \zeta(7) \Big) \\
& \overline{\zeta}_8 = \sum_{\substack{a_1, \ldots, a_8 \ne 0 \\ a_1 + \ldots + a_8 = 0}} \frac{1}{\left| a_1 \cdots a_8 \right|}
=
1680 \Big( \zeta(2)^4 + 48 \zeta(2) \zeta(3)^2 + 42 \zeta(2)^2 \zeta(4) + 147 \zeta(4)^2 \\
& + 288 \zeta(3) \zeta(5) + 248 \zeta(2) \zeta(6) + 762 \zeta(8) \Big) \\
& \overline{\zeta}_9 = \sum_{\substack{a_1, \ldots, a_9 \ne 0 \\ a_1 + \ldots + a_9 = 0}} \frac{1}{\left| a_1 \cdots a_9 \right|}
=
120960 \Big( \zeta(2)^3 \zeta(3) + 4 \zeta(3)^3 + 21 \zeta(2) \zeta(3) \zeta(4) \\
& + 9 \zeta(2)^2 \zeta(5) + 63 \zeta(4) \zeta(5) + 62 \zeta(3) \zeta(6) + 54 \zeta(2) \zeta(7) + 170 \zeta(9) \Big) \\
& \overline{\zeta}_{10} = \sum_{\substack{a_1, \ldots, a_{10} \ne 0 \\ a_1 + \ldots + a_{10} = 0}} \frac{1}{\left| a_1 \cdots a_{10} \right|}
=
30240 \Big( \zeta(2)^5 + 120 \zeta(2)^2 \zeta(3)^2 + 70 \zeta(2)^3 \zeta(4) \\
& + 840 \zeta(3)^2 \zeta(4) + 735 \zeta(2) \zeta(4)^2 + 1440 \zeta(2) \zeta(3) \zeta(5) + 2160 \zeta(5)^2 + 620 \zeta(2)^2 \zeta(6) \\
& + 4340 \zeta(4) \zeta(6) + 4320 \zeta(3) \zeta(7) + 3810 \zeta(2) \zeta(8) + 12264 \zeta(10)
 \Big)
\end{align*}

The $10$-th order sum has an approximate value of $1.1828 \times 10^9$. The two of the above values can be calculated exactly, using only $\pi$ constant:
\begin{align*}
\sum_{\substack{a_1, a_2 \ne 0 \\ a_1 + a_2 = 0}} \frac{1}{\left| a_1 a_2 \right|} & = \frac{\pi^2}{3} \\
\sum_{\substack{a_1, a_2, a_3, a_4 \ne 0 \\ a_1 + a_2 + a_3 + a_4 = 0}} \frac{1}{\left| a_1 a_2 a_3 a_4 \right|} & = \frac{19}{15} \pi^4
\end{align*}

It seems likely, that the above cases and along with $\overline{\zeta}_1 = 0$ are \emph{the only values} that are rational multiplicities of $\pi^n$ for any of $\overline{\zeta}_{n \ge 1}$, hence:

\begin{conjecture}
The only three cases for which $\overline{\zeta}_{n \ge 1}$ is a rational multiplicity of $\pi^n$ are $n = 1$, $n=2$ and $n=4$ with the coefficients $0$, $\frac{1}{3}$ and $\frac{19}{15}$ respectively.
\end{conjecture}

The second observation is that likely the number of elements in the sum is always equal to the number of partitions of $n$ into parts containing at least two elements, i.e. equal to $p(n) - p(n-1)$ where $p$ is the partition function. In other words, none of the coefficients in the polynomial is equal to zero. Hence:

\begin{conjecture}
The value of $\overline{\zeta}_{n \ge 1}$ is a homogeneous polynomial with $p(n) - p(n-1)$ monomials where each monomial is a product of zeta function values with a total weight equal to $n$.
\end{conjecture}

The relation between $\overline{\zeta}_n$ and $\zeta(n)$ can be inverted giving the following formulas:
\begin{align*}
2 \zeta(2) = & \overline{\zeta}_2 \\
12 \zeta(3) = & \overline{\zeta}_3 \\
84 \zeta(4) = & \overline{\zeta}_4 - 3 \overline{\zeta}_2^2 \\
720 \zeta(5) = & \overline{\zeta}_5 - 10 \overline{\zeta}_2 \overline{\zeta}_3 \\
7440 \zeta(6) = & \overline{\zeta}_6 - 15 \overline{\zeta}_2 \overline{\zeta}_4 - 10 \overline{\zeta}_3^2 + 30 \overline{\zeta}_2^3 \\
90720 \zeta(7) = & \overline{\zeta}_7 - 21 \overline{\zeta}_2 \overline{\zeta}_5 - 35 \overline{\zeta}_3 \overline{\zeta}_4 + 210 \overline{\zeta}_2^2 \overline{\zeta}_3 \\
1280160 \zeta(8) = & \overline{\zeta}_8 - 28 \overline{\zeta}_2 \overline{\zeta}_6 - 56 \overline{\zeta}_3 \overline{\zeta}_5 - 35 \overline{\zeta}_4^2 + 420 \overline{\zeta}_2^2 \overline{\zeta}_4 + 560 \overline{\zeta}_2 \overline{\zeta}_3^2 - 630 \overline{\zeta}_2^4 \\
20563200 \zeta(9) = & \overline{\zeta}_9 - 36 \overline{\zeta}_2 \overline{\zeta}_7 - 84 \overline{\zeta}_3 \overline{\zeta}_6 - 126 \overline{\zeta}_4 \overline{\zeta}_5 + 756 \overline{\zeta}_2^2 \overline{\zeta}_5 + 560 \overline{\zeta}_3^3 - 7560 \overline{\zeta}_2^3 \overline{\zeta}_3 + 2520 \overline{\zeta}_2 \overline{\zeta}_3 \overline{\zeta}_4 \\
370863360 \zeta(10) = & \overline{\zeta}_{10} - 45 \overline{\zeta}_2 \overline{\zeta}_8 - 120 \overline{\zeta}_3 \overline{\zeta}_7 -210 \overline{\zeta}_4 \overline{\zeta}_6 + 1260 \overline{\zeta}_2^2 \overline{\zeta}_6 - 126 \overline{\zeta}_5^2 + 3150 \overline{\zeta}_2 \overline{\zeta}_4^2 + 4200 \overline{\zeta}_3^2 \overline{\zeta}_4 \\
& - 18900 \overline{\zeta}_2^3 \overline{\zeta}_4 - 37800 \overline{\zeta}_2^2 \overline{\zeta}_3^2 + 22680 \overline{\zeta}_2^5 + 5040 \overline{\zeta}_2 \overline{\zeta}_3 \overline{\zeta}_5
\end{align*}

Finally, the last observation is that it seems likely that the above formulas can be generalized for arbitrary $n$, hence:
\begin{conjecture}
Every zeta value $\zeta(n)$ for $n \ge 2$ can be written as a $Q$-linear combination of $\overline{\zeta}_{n \ge 2}$ values.
\end{conjecture}


\printbibliography

@article{Bachmann2021,
  author  = {Bachmann, Henrik},
  title   = {Finite and symmetric {M}ordell--{T}ornheim multiple zeta values},
  journal = {Journal of Number Theory},
  year    = {2021},
  volume  = {221},
  pages   = {161--174},
}

@book{Jordan1950,
  author = {Jordan, Charles},
  address = {New York},
  booktitle = {Calculus of finite differences},
  edition = {2nd ed.},
  language = {eng},
  publisher = {Chelsea},
  title = {Calculus of finite differences},
  year = {1950},
}

@book{Andrews1999,
  title={Special Functions},
  author={Andrews, George E. and Askey, Richard and Roy, Ranjan},
  series={Encyclopedia of Mathematics and its Applications},
  volume={71},
  year={1999},
  publisher={Cambridge University Press},
  address={Cambridge},
}

@article{Tornheim1950,
  author  = {Tornheim, Leonard},
  title   = {Harmonic double series},
  journal = {American Journal of Mathematics},
  year    = {1950},
  volume  = {72},
  number  = {2},
  pages   = {303--314},
}

@article{Mordell1958,
  author  = {Mordell, Louis J.},
  title   = {On the evaluation of some multiple series},
  journal = {Journal of the London Mathematical Society},
  year    = {1958},
  volume  = {33},
  number  = {3},
  pages   = {368--371},
}

@inproceedings{Matsumoto2003,
  author    = {Matsumoto, Kohji},
  title     = {On {M}ordell--{T}ornheim and other multiple zeta-functions},
  booktitle = {Proceedings of the Session in Analytic Number Theory and Diophantine Equations},
  year      = {2003},
  volume    = {360},
  series    = {Bonner Mathematische Schriften},
  publisher = {University of Bonn},
  address   = {Bonn},
  pages     = {17}
}

@article{MatsumotoEtal2008,
  author  = {Matsumoto, Kohji and Nakamura, Takashi and Ochiai, Hiroyuki and Tsumura, Hirofumi},
  title   = {On value-relations, functional relations and singularities of {M}ordell--{T}ornheim and related triple zeta-functions},
  journal = {Acta Arithmetica},
  year    = {2008},
  volume  = {132},
  number  = {2},
  pages   = {99--125},
}

@article{MatsumotoTsumura2006,
  author    = {Matsumoto, Kohji and Tsumura, Hirofumi},
  title     = {On {W}itten multiple zeta-functions associated with semisimple {L}ie algebras {I}},
  journal   = {Annales de l'institut Fourier},
  year      = {2006},
  volume    = {56},
  number    = {5},
  pages     = {1457--1504},
}

\end{document}